\documentclass[fms,times]{cuparticle}
\usepackage{latexsym,amssymb}
\usepackage{amsmath, amscd, graphics, xypic, mathrsfs, setspace, bm, enumitem}

\newcommand{\Spec}{\operatorname{Spec}}

\newcommand{\isomt}{{\stackrel{{\scriptscriptstyle{\sim}}}{\;\rightarrow\;}}}

\renewcommand{\O}{{\mathcal O}}

\newcommand{\cplx}{{\mathbb C}}
\newcommand{\Q}{{\mathbb Q}}
\newcommand{\Z}{{\mathbb Z}}

\newcommand{\aone}{{\mathbb A}^1}
\newcommand{\pone}{{\mathbb P}^1}

\newcommand{\gm}[1]{{{\mathbf G}_{m}^{#1}}}
\renewcommand{\L}{{\mathcal L}}

\newcommand{\bpi}{\bm{\pi}}

\newcommand{\K}{{{\mathbf K}}}

\newcounter{intro}
\newtheorem{thm}{Theorem}[subsection]
\newtheorem{lem}[thm]{Lemma}
\newtheorem{cor}[thm]{Corollary}
\newtheorem{prop}[thm]{Proposition}

\newtheorem{conj}[thm]{Conjecture}

\newtheorem{thmintro}{Theorem}

\newtheorem{questionintro}[thmintro]{Question}

\newdefinition{definition}{Definition}
\newdefinition{ex}[thm]{Example}
\newdefinition{rem}[thm]{Remark}
\newdefinition{remintro}[thmintro]{Remark}

\newproof{proof}{Proof}

\volume{??}
\doi{??}

\begin{document}

\authorheadline{Aravind Asok, Jean Fasel and Michael J. Hopkins}
\runningtitle{Obstructions to algebraizing topological vector bundles}

\begin{frontmatter}

\title{Obstructions to algebraizing topological vector bundles}
\author[1]{A.~Asok}

\address[1]{Department of Mathematics, University of Southern California, 3620 S. Vermont Ave.,
  Los Angeles, CA 90089-2532, United States \ead{asok@usc.edu}}

\author[2]{J.~Fasel}

\address[2]{Institut Fourier - UMR 5582, Universit\'e Grenoble Alpes CS 40700, 38058 Grenoble Cedex 09; France \ead{jean.fasel@gmail.com}}

\author[3]{M.J.~Hopkins}

\address[3]{Department of Mathematics, Harvard University, One Oxford Street, Cambridge, MA 02138, United States \ead{mjh@math.harvard.edu}}

\received{12 October 2015}

\begin{abstract}
Suppose $X$ is a smooth complex algebraic variety.  A necessary condition for a complex topological vector bundle on $X$ (viewed as a complex manifold) to be algebraic is that all Chern classes must be algebraic cohomology classes, i.e., lie in the image of the cycle class map.  We analyze the question of whether algebraicity of Chern classes is sufficient to guarantee algebraizability of complex topological vector bundles.  For affine varieties of dimension $\leq 3$, it is known that algebraicity of Chern classes of a vector bundle guarantees algebraizability of the vector bundle.  In contrast, we show in dimension $\geq 4$  that algebraicity of Chern classes is insufficient to guarantee algebraizability of vector bundles.  To do this, we construct a new obstruction to algebraizability using Steenrod operations on Chow groups.  By means of an explicit example, we observe that our obstruction is non-trivial in general.
\end{abstract}

\end{frontmatter}
\section{Introduction}
Suppose $X$ is a smooth complex algebraic variety.  We write $X^{an}$ for $X(\cplx)$ viewed as a complex manifold.  Write $\mathscr{V}_n(X)$ for the set of isomorphism classes of rank $n$ algebraic vector bundles on $X$, $\mathscr{V}^{an}(X)$ for the set of isomorphism classes of rank $n$ analytic vector bundles on $X^{an}$ and $\mathscr{V}^{top}(X)$ for the set of isomorphism classes of rank $n$ complex topological vector bundles on $X^{an}$.

For any integer $n \geq 0$, the assignment $X \mapsto X^{an}$ gives rise to a sequence of functions
\[
\mathscr{V}_n(X) \longrightarrow \mathscr{V}^{an}(X) \longrightarrow \mathscr{V}^{top}(X).
\]
An element of $\mathscr{V}^{top}(X)$ that lies in the image of the composite map $\mathscr{V}_n(X) \longrightarrow \mathscr{V}^{top}(X)$ will be called an {\em algebraizable} vector bundle.  The motivating problem of this paper is: characterize algebraizable vector bundles among topological vector bundles.  This problem is very old; it is studied explicitly for projective varieties of small dimension, for example, in work of Schwarzenberger \cite{Schwarzenberger}, Atiyah--Rees \cite{AtiyahRees} and B{\u{a}}nic{\u{a}}--Putinar \cite{BanicaPutinar}.  Very little is known about this problem for varieties of dimension $\geq 4$.

Suppose $\mathcal{E}^{top} \to X^{an}$ is a complex topological vector bundle.  If $\mathcal{E}^{top}$ is algebraizable, then the Chern classes $c_i^{top}(\mathcal{E}^{top}) \in H^{2i}(X^{an},\Z)$ of $\mathcal{E}^{top}$ are algebraic, i.e., they lie in the image of the cycle class map
\[
cl: CH^i(X) \longrightarrow H^{2i}(X^{an},\Z).
\]
Using this observation, one can show that the map $\mathscr{V}_n(X) \longrightarrow \mathscr{V}^{top}(X)$ is, in general, neither injective nor surjective.  If $X$ is a smooth affine variety, then by Grauert's Oka-principle (\cite[\S 2 Satz I, II]{Grauert} or \cite[Theorem 7.2.1]{Forstneric} for a textbook treatment), every topological vector bundle on $X^{an}$ admits a unique analytic structure, i.e., the map $\mathscr{V}^{an}(X) \to \mathscr{V}^{top}(X)$ is a bijection.  The following question is a concrete form of the problem stated above.

\begin{questionintro}
\label{questionintro:griffiths}
If $X$ is a smooth complex affine variety, and $\mathcal{E}^{an} \to X^{an}$ is a complex analytic vector bundle with algebraic Chern classes, then is $\mathcal{E}^{an}$ is algebraizable?
\end{questionintro}

Question \ref{questionintro:griffiths} has a positive answer in small dimensions.  Serre's splitting theorem \cite[Th{\'e}or{\`e}me 1]{Serre} implies that any algebraic vector bundle of rank $r > \dim X$ on a smooth affine variety can be written as the direct sum of a vector bundle of rank $\leq \dim X$ and a trivial bundle (note: the smoothness hypothesis is unnecessary to apply Serre's result).  To answer Question \ref{questionintro:griffiths}, it therefore suffices to establish that topological vector bundles of rank below the dimension with algebraic Chern classes are algebraizable.

For vector bundles of rank $1$ on smooth affine varieties of any dimension, since $CH^1(X) = Pic(X)$ algebraicity of the Chern classes essentially by definition guarantees algebraizability; it follows immediately that Question \ref{questionintro:griffiths} has a positive answer for varieties of dimension $1$.  In dimension $2$, a positive answer to Question \ref{questionintro:griffiths} follows from the work of Murthy and Swan \cite{MurthySwan} who show that if $X$ is a smooth complex affine surface, then for any pair $(c_1,c_2) \in CH^1(X) \times CH^2(X)$ there is an algebraic vector bundle of rank $2$ with those Chern classes.  Similarly, in dimension $3$, a positive answer to Question \ref{questionintro:griffiths} follows from work of Mohan Kumar and Murthy \cite[Theorem 2.1]{MohanKumarMurthy}.  In particular, Mohan Kumar and Murthy established existence of algebraic vector bundles of rank $\leq 3$ with arbitrary prescribed Chern classes on smooth affine threefolds.

The main result of this paper is that Question \ref{questionintro:griffiths} admits a negative answer in the first unknown case: rank $2$ vector bundles on smooth complex affine varieties of dimension $\geq 4$.  To see this, we will give a necessary and sufficient condition for algebraizability of rank $2$ bundles on smooth complex affine $4$-folds involving the integral Steenrod squaring operation $Sq^2$ on Chow groups; this operation was described by Voevodsky \cite{VRed} and Brosnan \cite{Brosnan} though for our purposes the latter (more elementary) description is sufficient.  We then show, by means of explicit examples, that the necessary and sufficient condition for algebraizability we write down is not always satisfied. More precisely, we establish the following results (the first result is established just after Theorem \ref{thm:main} in the body of the text).

\begin{thmintro}
\label{thmintro:maintheorem}
Suppose $X$ is a smooth complex affine variety of dimension $4$, and $\mathcal{E}^{an} \to X^{an}$ is a rank $2$ complex analytic vector bundle with Chern classes $c_i^{top} \in H^{2i}(X^{an},\Z)$.  Assume the Chern classes $c_i^{top}$ of $\mathcal{E}^{an}$ are algebraic, i.e., lie in the image of the cycle class map $cl$.  The bundle $\mathcal{E}^{an}$ is algebraizable if and only if we may find $(c_1,c_2) \in CH^1(X) \times CH^2(X)$ with $(cl(c_1),cl(c_2)) = (c_1^{top},c_2^{top})$ such that $Sq^2c_2 + c_1 \cup c_2 = 0 \in CH^3(X)/2$; .
\end{thmintro}

\begin{remintro}
Note that, even if $c_1^{top}$ is zero in the above statement, unless the cycle class map $cl: CH^1(X) \to H^2(X,\Z)$ is injective in degree $1$ we cannot guarantee that $c_1$ can be chosen to vanish.  {\em A priori} it is possible that $Sq^2 c_2 + c_1 \cup c_2$ is always zero, but the following result shows that this is not the case.
\end{remintro}

\begin{thmintro}[See Corollary \ref{cor:mainexample}]
\label{thmintro:mainexample}
There exists a smooth hypersurface $Z$ of bidegree $(3,4)$ in $\pone \times {\mathbb P}^3$, such that, setting $X := (\pone \times {\mathbb P}^3)\setminus Z$, the following statements hold.
\begin{itemize}[noitemsep,topsep=1pt]
\item[(a)] The cycle class map $cl: CH^i(X) \to H^{2i}(X^{an},\Z)$ is injective for $i \leq 2$, and
\item[(b)] the manifold $X^{an}$ carries a rank $2$ topological vector bundle $\mathcal{E}^{an}$ with Chern classes $(0,c_2^{top})$ such that $c_2^{top}$ is algebraic and the unique lift $c_2$ of $c_2^{top}$ granted by (a) satisfies $Sq^2 c_2 \neq 0$ (in particular, $\mathcal{E}^{an}$ is not algebraizable).
\end{itemize}
\end{thmintro}

Theorem \ref{thmintro:maintheorem} is established by first observing that the map $\mathscr{V}_n(X) \to \mathscr{V}_n^{top}(X)$ factors through $[X,BGL_n]_{\aone}$, i.e., the set of $\aone$-homotopy classes of maps $X$ to $BGL_n$ (the set $[X,BGL_n]_{\aone}$ has been called the set of {\em motivic vector bundles on $X$} by the authors).  Thus, to produce the obstruction, it suffices to obstruct existence of an $\aone$-homotopy class of maps, i.e., to obstruct existence of a ``motivic" lift of a given homotopy class.  This is accomplished by analysis of the Moore--Postnikov factorization of the map assigning to the universal rank $2$ vector bundle over the Grassmannian its Chern classes; these ideas are discussed in Section \ref{ss:obstructiontheory}.  The primary obstruction to existence of a lifting yields the condition of the statement, and we establish a vanishing theorem showing that, under suitable hypotheses, all higher obstructions vanish.  To prove necessity of the vanishing of the obstruction, we appeal to Morel's vector bundle classification: if $X$ is smooth and affine, then $[X,BGL_n]_{\aone} \cong \mathscr{V}_n(X)$ (see \cite[Theorem 1]{AsokHoyoisWendt}).

The construction of the example in Theorem \ref{thmintro:mainexample} is closely related with the failure of the integral Hodge conjecture for the hypersurface $Z$.  The failure of injectivity of the cycle class map $CH^3(X) \to H^{6}(X^{an},\Z)$ is precisely what allows the examples above to exist.  In Section \ref{ss:noritotaro}, we explain how some general conjectures on algebraic cycles suggest examples like those above are ``generic".  By considering products of the form $X \times {\mathbb A}^n$ with $X$ as in Theorem \ref{thmintro:mainexample}, one may construct examples of non-algebraizable topological vector bundles with algebraic Chern classes in any dimension $\geq 4$.

\begin{remintro}
Theorem \ref{thmintro:mainexample} also provides a counterexample to a related K-theoretic variant of Question \ref{questionintro:griffiths}.  Indeed, for any smooth $\cplx$-scheme $X$, we may consider the Grothendieck groups $K_0(X)$ and $K_0^{top}(X^{an})$.  The functions $\mathscr{V}_n(X) \to \mathscr{V}_n^{top}(X^{an})$ (for varying $n$) induce a function $K_0(X) \to K_0^{top}(X)$; we will say that a class in $K_0^{top}(X)$ is algebraic if it lies in the image of this map.  One might ask whether topological vector bundles whose associated K-theory class is algebraic themselves admit algebraic structures.

Since Chern classes factor through K-theory, algebraicity of the topological K-theory class of a vector bundle is a stronger restriction than algebraicity of Chern classes.  \cite[Theorem 2.1]{MohanKumarMurthy} essentially shows that the K-theoretic variant of Question \ref{questionintro:griffiths} admits a positive solution for smooth affine $\cplx$-schemes of dimension $\leq 3$.  On the other hand, the proof of Proposition \ref{prop:relevanthomotopysheaves} shows that the restriction on Chern classes appearing in Theorem \ref{thmintro:maintheorem} is precisely the primary obstruction to building a rank $2$ vector bundle on a smooth affine $4$-fold given a fixed class in $K_0(X)$.  Thus in dimension $\geq 4$, the K-theoretic variant of Question \ref{questionintro:griffiths} mentioned in the previous paragraph also admits a negative solution.
\end{remintro}

\subsubsection*{Acknowledgements}
The first named author would like to thank Burt Totaro for some helpful correspondence as regards the integral Hodge conjecture on threefolds and for pointing out the paper \cite{Totaro}.  Asok was supported by National Science Foundation Award DMS-1254892.  The third-named author would like to thank Philip Griffiths for correspondence related to the paper \cite{Griffiths}, which was a source of inspiration for the questions considered here.  Hopkins was supported by National Science Foundation Award DMS-0906194. Finally, we thank Oliver R\"ondigs for some comments on the first version of this paper, and the referees for some helpful comments and questions.

\section{Obstruction theory: proof of Theorem \ref{thmintro:maintheorem}}
The goal of this section is to prove Theorem \ref{thmintro:maintheorem}.  We begin with some preliminaries regarding Chern classes and obstruction theory in $\aone$-homotopy theory.  We will use rather freely the terminology of \cite{MV,MField}; rather than bulking up this paper with a long section of notation and terminology, we have chosen to focus on the argument; we will follow the notations and conventions of \cite[\S 2.1]{AsokFaselA3minus0} and we refer the reader there for terminology not explicitly defined here (e.g., spaces, homotopy sheaves, strong and strict $\aone$-invariance of sheaves and their basic properties).  When we consider cohomology of a sheaf on a smooth scheme, we mean cohomology in the Nisnevich site.  We also remind the reader that strongly or strictly $\aone$-invariant sheaves are {\em unramified} and that to check a morphism of unramified sheaves is an isomorphism, it suffices to check this on sections over finitely generated extensions of the base field.  For a general discussion of the Moore--Postnikov factorization in $\aone$-homotopy theory, we refer the reader to \cite[\S 6]{AsokFaselA3minus0}.

\subsection{Chern classes and the basic obstruction theory problem}
\label{ss:obstructiontheory}
Write $BGL_n$ for the simplicial classifying space of $GL_n$ (see, e.g., \cite[\S 4.1]{MV}).  The determinant map $GL_n \to \gm{}$ induces a map of classifying spaces $BGL_n \to B\gm{}$.  If $X \to BGL_n$ is a simplicial homotopy class of maps representing a rank $n$ vector bundle on a smooth scheme $X$, then the composite map $X \to BGL_n \to B\gm{}$ represents the determinant line bundle of this vector bundle.  Some $\aone$-homotopy sheaves of $BGL_n$ were computed in \cite[Theorem 7.20]{MField}; we refer the reader to \cite[Lemma 3.1, Theorem 3.2]{AsokFaselSpheres} for convenient references in the form we require.

\begin{lem}
\label{lem:lowdegreehomotopysheavesofgln}
There are canonical isomorphisms $\bpi_1^{\aone}(BGL_n) \isomt \gm{}$ induced by the determinant and $\bpi_2^{\aone}(BGL_2) \cong \K^{MW}_2$.
\end{lem}

The motivic cohomology of $BGL_n$ is a polynomial algebra over the motivic cohomology of a point in variables $c_1,\ldots,c_n$ where $c_i$ has bidegree $(2i,i)$ (this is ``well-known", but see, e.g., \cite[Proposition 2]{PushinChern} for a precise statement; Pushin's argument is a version of the argument of \cite{GilletChern}, which itself goes back to Grothendieck's axiomatic treatment of Chern classes).  By Voevodsky's (unstable) $\aone$-representability of motivic cohomology \cite[\S 2.3 Theorem 2]{DeligneVoevodsky}, each $c_i$ corresponds to an $\aone$-homotopy class of maps
\[
c_i: BGL_n \to K(\Z(i),2i).
\]
There is an $\aone$-weak equivalence $Gr_n \to BGL_n$, where $Gr_n$ is the infinite Grassmannian \cite[\S4 Proposition 3.7]{MV}. Here and henceforth, we can fix suitable $\aone$-fibrant models of $BGL_n$ and $K(\Z(i),2i)$ so that the $\aone$-homotopy classes of maps above are represented by actual morphisms of spaces $Gr_n \to K(\Z(i),2i)$.

Morel and Voevodsky introduced a complex realization functor \cite[\S 3.3]{MV}.  Under complex realization $Gr_n$ is sent to the usual complex Grassmannian.  The model of $K(\Z(i),2i)$ in terms of effective cycles \cite[\S 6]{VZeroslice} and the classical Dold--Thom theorem show that the complex realization of $K(\Z(i),2i)$ is $K(\Z,2i)$.  Moreover, the maps $c_i: Gr_n \to K(\Z(i),2i)$ are sent by realization to the usual Chern class maps $c_i^{top}: Gr_n \to K(\Z,2i)$.  Indeed, this observation is a consequence of (1) the fact that the finite-dimensional Grassmannian varieties $Gr_{n,N}$ admit a cellular decomposition \cite[Example 1.9.1]{Fulton} and thus the cycle class map from Chow groups to ordinary cohomology is an isomorphism (see also \cite[Proposition 4.4]{DuggerIsaksenCellular} for a more homotopic statement), (2) the fact that $Gr_n$ is a filtered colimit of the finite-dimensional Grassmannian varieties $Gr_{n,N}$ by construction, and (3) the fact that motivic cohomology of $Gr_{n,N}$ in any given degree stabilizes for $N$ large enough, e.g., by Totaro's argument \cite[Definition-Proposition 1 and \S 2.7]{EdidinGraham}.  We use the compatibilities mentioned above without mention in the sequel.

The space $K(\Z(n),2n)$ is not an Eilenberg--Mac Lane space in the sense that its $\aone$-homotopy sheaves are not, in general, concentrated in a single degree.  Nevertheless, one may identify the $\aone$-homotopy sheaves of $K(\Z(n),2n)$.  The first statement is simply a reformulation of the $\aone$-representability of motivic cohomology mentioned above, while the second statement is a reformulation of the Nesterenko--Suslin--Totaro theorem (see, e.g., \cite[Theorem 5.1]{MVW}).

\begin{lem}
\label{lem:homotopysheavesofeilenbergmaclanespaces}
There is a canonical isomorphism $\bpi_i^{\aone}(K(\Z(n),2n)) \cong {\mathbf H}^{2n-i,n}$, where $\mathbf{H}^{2n-i,n}$ is the sheafification of the presheaf $U \mapsto H^{2n-i,n}(U,\Z)$.  In particular, $K(\Z(n),2n)$ is $\aone$-$(n-1)$-connected and its first non-vanishing $\aone$-homotopy sheaf is $\K^M_n$, the $n$-th unramified Milnor K-theory sheaf.
\end{lem}

Ideally, we would study the existence of a vector bundle of rank $2$ with given Chern classes $(c_1,c_2)$ by studying an obstruction theory problem deduced from the Moore--Postnikov factorization of the map $(c_1,c_2): BGL_2 \to K(\Z(1),2) \times K(\Z(2),4)$.  The fact that $K(\Z(n),2n)$ is not an Eilenberg--Mac Lane space causes various technical complications and we use the description of its homotopy sheaves provided above to produce an equivalent yet technically simpler problem.

The first stage of the $\aone$-Postnikov tower of $K(\Z(n),2n)$ yields, by appeal to Lemma \ref{lem:homotopysheavesofeilenbergmaclanespaces}, a canonical map $K(\Z(n),2n) \to K(\K^M_n,n)$.  In particular, composition of the universal $n$-th Chern class with this map yields a map:
\begin{equation}
\label{eqn:kcohomologychernclasses}
c_n': BGL_r \to K(\K^M_n,n)
\end{equation}
for any $r \geq n$; this is a modified version of the universal $n$-th Chern class.

\begin{rem}
\label{rem:firstchern}
Since $\K^M_1 = \gm{}$ and since the motivic complex $\Z(1)$ is $\gm{}$ concentrated in a single degree \cite[Theorem 4.1]{MVW}, the map $K(\Z(1),2) \to K(\K^M_1,1)$ induced by the Postnikov tower is a simplicial weak equivalence.  In particular, the maps $c_1$ and $c_1'$ lie in the same simplicial homotopy class.  Explicitly identifying the homotopy sheaves of $\bpi_i^{\aone}(K(\Z(n),2n))$ in all degrees seems at the moment intractable: the Beilinson--Soul\'e vanishing conjecture predicts that $K(\Z(n),2n)$ is $\aone$-$(2n-1)$-truncated, i.e., its homotopy sheaves vanish for $i > 2n-1$.
\end{rem}

\subsection{Homotopy sheaves and Moore--Postnikov factorizations}
\label{ss:identifyingobstructions}
Consider the map
\[
(c_1',c_2'): BGL_2 \to K(\K^M_1,1) \times K(\K^M_2,2),
\]
where $c_1'$ and $c_2'$ are the maps mentioned in \ref{eqn:kcohomologychernclasses}.  Our goal is to analyze the homotopy fiber of this map.  Write $\mathscr{F}_2$ for the $\aone$-homotopy fiber of $(c_1',c_2')$.

By Lemma \ref{lem:lowdegreehomotopysheavesofgln} there is a canonical identification $\bpi_1^{\aone}(BGL_2) \cong \gm{}$ and the map $c_1'$ yields an isomorphism of $\aone$-fundamental sheaves of groups (see also Remark \ref{rem:firstchern}).  As a consequence, we will be considering a twisted obstruction theory problem and we will have to keep track of the action of $\gm{}$ on the higher $\aone$-homotopy sheaves of the $\aone$-homotopy fiber.  We begin by identifying the relevant $\aone$-homotopy sheaves of $\mathscr{F}_2$ together with the action of $\gm{}$ (for the relevant definitions about twisted $\aone$-homotopy sheaves, we refer the reader to \cite[\S 2.4]{AsokFaselA3minus0}).

\begin{prop}
\label{prop:relevanthomotopysheaves}
There are canonical isomorphisms $\bpi_2^{\aone}(\mathscr{F}_2) \cong \mathbf{I}^3$ and $\bpi_3^{\aone}(\mathscr{F}_2) \cong \bpi_3^{\aone}(BGL_2) \cong \bpi_2^{\aone}(SL_2)$.  Moreover, the actions of $\gm{}$ on these sheaves coincide with the actions described in \textup{\cite[Propositions 6.3 and 6.5]{AsokFaselThreefolds}}.
\end{prop}

\begin{proof}
Observe that $K(\K^M_1,1) \times K(\K^M_2,2)$ is $\aone$-$2$-truncated, i.e., has $\aone$-homotopy sheaves concentrated in degrees $\leq 2$.  The second assertion is therefore immediate from the isomorphism $\bpi_3^{\aone}(BGL_2) \cong \bpi_2^{\aone}(SL_2)$.

For the first statement, the long exact sequence in $\aone$-homotopy sheaves of a fibration yields (using Lemma \ref{lem:lowdegreehomotopysheavesofgln}) a short exact sequence
\[
0 \longrightarrow \bpi_2^{\aone}(\mathscr{F}_2) \longrightarrow \K^{MW}_2 \longrightarrow \K^{M}_2 \longrightarrow 0.
\]
On the other hand, recall that there is an exact sequence of sheaves of the form
\[
0 \longrightarrow \mathbf{I}^3 \longrightarrow \K^{MW}_2 \longrightarrow \K^M_2 \longrightarrow 0;
\]
this is the sheafified version of the exact sequence of \cite[Corollaire 5.4]{MorelKMW}.  We will identify these two exact sequences.

To this end, let us unwind some of the identifications.  Recall that the $\aone$-fiber sequence
\[
\gm{} \longrightarrow BSL_2 \longrightarrow BGL_2
\]
yields the identification $\bpi_2^{\aone}(BSL_2) \cong \bpi_2^{\aone}(BGL_2)$.

Identify $SL_2 \cong Sp_2$ and consider the following commutative diagram:
\[
\xymatrix{
BSp_2 \ar[r]\ar[d] & BGL_2 \ar[d] \\
BSp_{\infty} \ar[r] & BGL_{\infty}.
}
\]
The left vertical map induces an isomorphism on $\bpi_2^{\aone}$ by \cite[Theorem 2.6]{AsokFaselThreefolds} and yields the isomorphism $\bpi_2^{\aone}(BSp_2) \cong \bpi_2^{\aone}(BSp_{\infty}) \cong \mathbf{KSp}_2$.  By Suslin's theorem $\mathbf{KSp}_2 \cong \K^{MW}_2$ \cite[Corollaries 6.2,6.4 and Theorem 6.5]{Suslin87} (see \cite[Theorem 4.1.2]{AsokFaselKODegree} and the discussion there for an explanation in the context we consider).  On the other hand, the map on $\bpi_2^{\aone}$ induced by the right vertical map is a map $\K^{MW}_2 \to \K^M_2$ (see \cite[Lemma 3.1 and Theorem 3.2]{AsokFaselSpheres}).  The bottom horizontal map induces the forgetful map which coincides with the standard surjection $\K^{MW}_2 \to \K^M_2$.

The second Chern class map $BGL_2 \to K(\K^M_2,2)$ necessarily factors through the map $BGL_2 \to BGL_{\infty}$ and the map $BGL_{\infty} \to K(\K^M_2,2)$ is an isomorphism on $\bpi_2^{\aone}$ by Matsumoto's theorem identifying Quillen's $K_2$ of a field with Milnor's $K_2$.  Therefore, the second Chern class map induces a map on $\bpi_2^{\aone}$ that, up to the identifications described above, coincides with the forgetful map $\K^{MW}_2 \to \K^M_2$, which is precisely what is asserted above.

By construction, the actions mentioned in the statement are inherited from the $\gm{}$-actions on homotopy sheaves of $BGL_2$ as described in the referenced propositions.
\end{proof}

\begin{thm}
\label{thm:main}
Suppose $X$ is a smooth affine $4$-fold over an algebraically closed field having characteristic unequal to $2$, and fix a line bundle $\L$ on $X$ with first Chern class $c_1 \in CH^1(X)$.  Given $c_2 \in CH^2(X)$, the pair $(c_1,c_2)$ are the first and second Chern classes of a vector bundle of rank $2$ and determinant $\L$ if and only if $Sq^2 c_2 + c_1 \cup c_2 = 0$.
\end{thm}

\begin{proof}
Suppose $X$ is a smooth affine $4$-fold over $\Spec k$ and fix $(c_1,c_2) \in CH^1(X) \times CH^2(X)$ as in the statement.  Such a pair determines a map $\psi: X \to K(\K^M_1,1) \times \K(\K^M_2,2)$.  Fix a line bundle ${\mathcal L}$ on $X$ representing $c_1$.  We now analyze the Moore--Postnikov factorization of the map $(c_1,c_2'): BGL_2 \to K(\K^M_1,1) \times K(\K^M_2,2)$.

This analysis proceeds in several steps, which we now outline.  First, we analyze the primary obstruction: we begin by constructing a ``quotient" of the primary obstruction, which is easier to identify.  Then, we show that this ``quotient" of the primary obstruction actually coincides with the primary obstruction by means of suitable cohomological vanishing statements.  Next, we analyze the secondary obstruction.  We show the secondary obstruction vanishes again by establishing a general cohomological vanishing result.  Finally, we appeal to general theory to show that no further obstructions can arise.

{\bf Step 1: Analyzing the primary obstruction.}  The primary obstruction to lifting $\psi$ to a map $X \to BGL_2$ is, by means of Proposition \ref{prop:relevanthomotopysheaves}, a class in $H^3(X,\mathbf{I}^3({\mathcal L}))$.
\newline

{\bf Step 1a: A quotient of the primary obstruction.}  To identify this obstruction more explicitly, we consider the exact sequence of \cite[\S 2.1 p. 423]{FaselIJ}:
\[
0 \longrightarrow \mathbf{I}^{j+1}(\L) \longrightarrow \mathbf{I}^j(\L) \longrightarrow \mathbf{I}^j/\mathbf{I}^{j+1} \longrightarrow 0.
\]
By the sheafified version of the Milnor conjecture on quadratic forms \cite{OVV} there is a canonical isomorphism of sheaves $\mathbf{I}^j/\mathbf{I}^{j+1} \cong \K^M_j/2$.  We consider the composite map (via this isomorphism)
\[
H^2(X,\K^M_2) \longrightarrow H^3(X,\mathbf{I}^3(\L)) \longrightarrow H^3(X,\K^M_3/2).
\]
The first map here is precisely the $k$-invariant in the Moore--Postnikov factorization and the identifications of Proposition \ref{prop:relevanthomotopysheaves} show that it is exactly the connecting homomorphism in the long exact sequence in cohomology associated with the exact sequence of sheaves on $X$:
\[
0 \longrightarrow \mathbf{I}^3({\mathcal L}) \longrightarrow \K^{MW}_2({\mathcal L}) \longrightarrow \K^M_2 \longrightarrow 0.
\]
The composite above can be described in a fashion that extends a result of Totaro.  More precisely, we claim that if $c_2 \in H^2(X,\K^M_2)$, then the map just described
\[
H^2(X,\K^M_2) \longrightarrow H^3(X,\K^M_3/2)
\]
sends $c_2 \mapsto Sq^2 c_2+ c_1(\L) \cup c_2$.  First, we treat the case where $\mathcal{L}$ is trivial.  In that case, the map $H^2(X,\K^M_2) \longrightarrow H^3(X,\K^M_3/2)$ factors through the mod $2$ reduction map $CH^2(X) = H^2(X,\K^M_2) \to H^2(X,\K^M_2/2) = CH^2(X)/2$ (see \cite[Diagram 2.1.2]{AsokFaselSecondary} and the preceding discussion).  Therefore, \cite[Theorem 1.1]{TotaroWitt} implies that the composite map is precisely the composite of the mod $2$ reduction map and $Sq^2$, i.e., the integral Steenrod squaring map.  To treat the general case, we appeal to \cite[Theorem 3.4.1]{AsokFaselSecondary} (applied in the case $i = j = 2$), which reduces the description of the relevant operation when $\mathcal{L}$ is non-trivial to the case where $\mathcal{L}$ is trivial.\newline

{\bf Step 1b: the actual primary obstruction.}  The relationship between the obstruction just mentioned and the {\em actual} primary obstruction is measured by the difference between the groups $H^3(X,\mathbf{I}^3(\L))$ and $H^3(X,\K^M_3/2)$.  In this case, we have the long exact sequence
\[
\cdots \longrightarrow H^3(X,\mathbf{I}^4(\L)) \longrightarrow H^3(X,\mathbf{I}^3(\L)) \longrightarrow H^3(X,\K^M_3/2) \longrightarrow H^4(X,\mathbf{I}^4(\L)) \longrightarrow \cdots.
\]
Since $X$ is a smooth affine $4$-fold over an algebraically closed field, by \cite[Proposition 5.2]{AsokFaselThreefolds} the groups $H^3(X,\mathbf{I}^4(\L))$ and $H^4(X,\mathbf{I}^4(\L))$ vanish.  Thus, the primary obstruction coincides with the ``quotient" constructed in Step 1a, i.e., the primary obstruction to lifting is precisely the vanishing of $Sq^2 c_2 + c_1 \cup c_2$. \newline

{\bf Step 2: Analyzing the secondary obstruction.}  If the primary obstruction vanishes, then we can choose a lift to the second stage of the Moore--Postnikov factorization of the map $(c_1,c_2')$.  Upon choosing a lift to this stage, we obtain a secondary obstruction lying in $H^4(X,\bpi_3^{\aone}({\mathscr F}_2)(\L))$.  By means of Proposition \ref{prop:relevanthomotopysheaves}, this obstruction is an element of $H^4(X,\bpi_3^{\aone}(BGL_2)(\L))$.  A priori, this obstruction depends on the choice of lift, but we claim $H^4(X,\bpi_3^{\aone}(BGL_2)(\L)) = 0$, independent of this choice.

To see this, recall by \cite[Theorem 3.3]{AsokFaselThreefolds} that $\bpi_3^{\aone}(BSp_2)$ is an extension of $\K^{Sp}_3=\mathbf{GW}^2_3$ by a certain sheaf $\mathbf{T}'_4$.  Observe first that $H^4(X,\mathbf{GW}^2_3(\L)) = 0$ by explicit construction of the Gersten resolution.  Indeed, $H^4(X,\mathbf{GW}^2_3(\L))$ is simply a quotient of $\bigoplus_{x \in X^{(4)}} (\mathbf{GW}^2_3(\L))_{-4}(\kappa_x)$ and the latter vanishes by \cite[Lemma 4.11]{AsokFaselThreefolds}.

By \cite[Corollary 4.9]{AsokFaselThreefolds} the induced $\gm{}$-action on $\mathbf{T}'_4$ is trivial.  Again using \cite[Theorem 3.3]{AsokFaselThreefolds}, $\mathbf{S}'_4$ is a quotient of $\mathbf{T}'_{4}$ by $\mathbf{D}_5$; the sheaf $\mathbf{S}'_4$ is itself a quotient of $\K^M_4/12$ while the sheaf $\mathbf{D}_5$ is a quotient of $\mathbf{I}^5$.  Observe that $H^4(X,\K^M_4/12) \cong CH^4(X)/12$.  Since $X$ is smooth and affine over an algebraically closed field, $CH^4(X)$ is divisible and therefore, $H^4(X,\K^M_4/12)$ is trivial.  On the other hand, $H^4(X,\mathbf{I}^5)$ is trivial since \cite[Proposition 5.1]{AsokFaselThreefolds} shows the sheaf $\mathbf{I}^5$ is itself trivial when restricted to $X$.  \newline

{\bf Step 3. Lifting to $BGL_2$.} Since the secondary obstruction vanishes, we may choose a lift to the third stage of the Moore--Postnikov factorization of the map.  Since $X$ has Nisnevich cohomological dimension $4$, there are no further obstructions to lifting and we may arbitrarily make a choice of lift of $X$ to the $4$-th stage of the Moore--Postnikov factorization of the map $(c_1,c_2')$.  Moreover, again using the fact that $X$ has Nisnevich cohomological dimension $4$ we see that beyond the $4$-th stage of the factorization all lifts are uniquely determined.  Thus, by the same argument as \cite[Proposition 6.2]{AsokFaselThreefolds}, we obtain an element of $[X,BGL_2]_{\aone}$.  Combining the discussion of the previous steps we see that a necessary and sufficient condition to lift $(c_1,c_2) \in CH^1(X) \times CH^2(X)$ to an element of $[X,BGL_2]_{\aone}$ is the vanishing of $Sq^2 c_2 + c_1 \cup c_2 \in CH^3(X)/2$. \newline

{\bf Step 4. Geometrization.}  Finally, we apply Morel's $\aone$-representability theorem for vector bundles \cite[Theorem 1]{AsokHoyoisWendt} to identify $[X,BGL_2]_{\aone}$ with $\mathscr{V}_2(X)$.
\end{proof}

\begin{proof}[Proof of Theorem \ref{thmintro:maintheorem}]
Consider the map
\[
(c_1^{top},c_2^{top}) : BGL_2(\cplx) \longrightarrow K(\Z,2) \times K(\Z,4).
\]
The map $(c_1^{top},c_2^{top})$ is a $4$-equivalence and thus the homotopy fiber is $5$-connected.  If $X$ is a smooth affine $4$-fold, then $X^{an}$ has the homotopy type of a CW complex of dimension $\leq 4$.  A straightforward obstruction theory argument then shows that there is a bijection $\mathscr{V}_2^{top}(X) \isomt H^2(X^{an},\Z) \times H^4(X^{an},\Z)$.  Combining this observation with Theorem \ref{thm:main} yields Theorem \ref{thmintro:maintheorem}.
\end{proof}

\begin{rem}
It is possible to establish Theorem \ref{thmintro:maintheorem} by a direct analysis of the Moore--Postnikov factorization of $(c_1,c_2): BGL_2 \to K(\Z(1),2) \times K(\Z(2),4)$, but identifying the $\aone$-homotopy sheaves of the $\aone$-homotopy fiber is slightly more complicated than the approach described above.
\end{rem}

\section{Constructing examples: proof of Theorem \ref{thmintro:mainexample}}
The goal of this section is to construct explicit examples of smooth complex affine varieties for which (a) the cycle class map $CH^i(X) \to H^{2i}(X,\Z)$ is injective for $i \leq 2$ and (b) there exist classes $\alpha \in CH^2(X)$ such that $Sq^2 \alpha \neq 0$.

\subsection{Explicit examples of non-algebraizable vector bundles}
If $\dim X < 3$, $CH^3(X)$ is trivial by definition and if $\dim X = 3$ then $CH^3(X)$ is divisible (which is classical); thus, if $\dim X \leq 3$, then $CH^3(X)/2$ is trivial.  Therefore the first dimension that can support non-trivial examples of the kind we envision is dimension $4$.

\begin{prop}
\label{prop:cycleclassisomorphisminlowdegrees}
Suppose $Y$ is a smooth projective variety of dimension $\geq 4$ and $Z$ is an ample hypersurface on $Y$ (i.e., $\O_Y(Z)$ is an ample line bundle) and set $X:= Y \setminus Z$.  If the cycle class map $CH^i(Y) \to H^{2i}(Y^{an},\Z)$ is an isomorphism for $i\leq 2$, then the cycle class maps $CH^i(X) \to H^{2i}(X^{an},\Z)$ are injective for $i \leq 2$.
\end{prop}

\begin{proof}

Using the assumption on the dimension and the hypothesis that $Z$ is an ample hypersurface, the Grothendieck--Lefschetz theorem guarantees that the pullback map $Pic(Y) \to Pic(Z)$ is an isomorphism \cite[Corollary IV.3.3]{HartshorneAmple}.  By the usual Lefschetz hyperplane theorem \cite[Theorem 1.1]{VoisinHodgeTheoryII}, the map $H^i(Y^{an},\Z) \to H^i(Z^{an},\Z)$ is an isomorphism for $i \leq 2$ and injective for $i = 3$.  Since the cycle class map commutes with pullbacks \cite[Proposition 9.2.1(i)]{VoisinHodgeTheoryII} and since the map $Pic(Y) \to H^2(Y^{an},\Z)$ is an isomorphism by assumption, we conclude that $Pic(Z) \to H^2(Z^{an},\Z)$ is an isomorphism as well.

Since the cycle class map commutes with Gysin maps \cite[Proposition 9.2.1(ii)]{VoisinHodgeTheoryII}, there is a commutative diagram of the form
\[
\xymatrix{
CH^{i-1}(Z) \ar[r]\ar[d]& CH^{i}(Y) \ar[r]\ar[d] & CH^{i}(X)\ar[r]\ar[d] & 0 \ar[d] \\
H^{2i-2}(Z^{an},\Z) \ar[r] & H^{2i}(Y^{an},\Z) \ar[r] & H^{2i}(X^{an},\Z) \ar[r] & H^{2i-1}(Z^{an},\Z)
}
\]
The  map $cl: CH^{i-1}(Z) \to H^{2i-2}(Z^{an},\Z)$ is an isomorphism for $i = 1$ essentially by definition and an isomorphism for $i = 2$ by the discussion of the previous paragraph.  The map $cl: CH^i(Y) \to H^{2i}(Y^{an},\Z)$ is an isomorphism for $i \leq 2$ by assumption.  Since the right vertical map is injective, it follows from the five lemma that $CH^i(X) \to H^{2i}(X^{an},\Z)$ is injective for $i \leq 2$ as well.
\end{proof}

\begin{rem}
If $Y$ is $1$-connected, it follows from the Lefschetz theorem on the fundamental group (see, e.g., \cite[Corollary IV.2.2]{HartshorneAmple}) that $Z$ is also $1$-connected.  In that case, one concludes that $Pic(X) \to H^{2}(X^{an},\Z)$ is an isomorphism as well since $H^1(Z^{an},\Z) = 0$ by the Hurewicz theorem.  One way to guarantee that $Y$ is simply connected and that the cycle class map $CH^i(Y) \to H^{2i}(Y^{an},\Z)$ is an isomorphism is to require that $Y$ admit a cellular decomposition in the sense of, e.g., \cite[Example 1.9.1]{Fulton}.  Recall that a variety $Y$ is said to admit a cellular decomposition if it admits an increasing filtration by closed subvarieties $\emptyset = Y_{-1} \subset \cdots \subset Y_i \subset Y_{i+1} \subset \cdots \subset Y_n = Y$ such that $Y_i \setminus Y_{i-1}$ can be written as a disjoint union of varieties isomorphic to affine space.  In that case, $Y$ is a rational variety and $Y^{an}$ is simply connected, e.g., by \cite{Serreunirational}.  In the examples we consider below, $Y$ will admit a cellular decomposition.

If $Y$ admits a cellular decomposition, then the cohomology of $Y^{an}$ is concentrated in even degrees.  In particular, $H^5(Y^{an},\Z) = 0$ and therefore the map $H^4(X^{an},\Z) \to H^3(Z^{an},\Z)$ in the Gysin sequence is surjective.  By taking $Z$ to be a sufficiently ample hypersurface, we can guarantee that $H^3(Z^{an},\Z)$ is non-zero in general (e.g., take a smooth hypersurface of degree $\geq 5$ in ${\mathbb P}^4$).  Thus, in general, $CH^2(X) \to H^4(X^{an},\Z)$ need not be surjective.
\end{rem}

Fix an isomorphism $CH^*(\pone \times {\mathbb P}^3) \cong \Z[\xi,\tau]/\langle \xi^2,\mu^4 \rangle$ (here $\xi$ and $\mu$ are elements of degree $1$ in the Chow ring).  If $Z$ is a smooth hypersurface of bidegree $(d_1,d_2)$, then under this isomorphism $[Z] = d_1 \xi + d_2 \mu$.  We now identify the Chow groups of the complement of $Z$ in ${\pone} \times {\mathbb P}^3$ in low degrees; we do this since the computation is concrete and elementary (the full strength of this statement is not necessary in Theorem \ref{thm:mainexample} below).

\begin{prop}
\label{prop:restrictiondegree2nontrivial}
Suppose $Z \subset {\pone} \times {\mathbb P}^3$ is a hypersurface of bidegree $(d_1,d_2)$ with $d_1 \neq 0, d_2 \neq 0$.  Set $g = gcd(d_1,d_2)$ and pick $m$ and $n$ such that $md_1 + nd_2 = g$.  If $X:= {\pone} \times {\mathbb P}^3 \setminus Z$, then
\[
\begin{split}
CH^1(X) &\cong \Z/d_1\Z \oplus \Z/d_2\Z, \text{ and } \\
CH^2(X) &\cong \Z/g\Z \oplus \Z/{\scriptstyle \frac{d_2^2}{g}}\Z.
\end{split}
\]
Under the first isomorphism $\Z/d_1\Z$ is generated by the image of $\xi$ while $\Z/d_2\Z$ is generated by the image of $\mu$.  Under the second isomorphism, the class $\xi \mu$ is sent to $(1,-\frac{md_2}{g})$, while the class $\mu^2$ is sent to a generator of $\Z/{\scriptstyle \frac{d_2^2}{g}}\Z$.  If $d_1 \nmid d_2$ and $d_2 \nmid d_1$ (e.g., if $gcd(d_1,d_2) = 1$), then $\xi\mu$ can be assumed to restrict non-trivially to $CH^2(X)$.
\end{prop}

\begin{proof}
It suffices to identify the map $CH^{j-1}(Z) \to CH^j(\pone \times {\mathbb P}^3)$.  Since $i^*: CH^{j-1}(\pone \times {\mathbb P}^3) \to CH^{j-1}(Z)$ is an isomorphism for $j \leq 2$ as observed in the proof of Proposition \ref{prop:cycleclassisomorphisminlowdegrees}, it remains to compute the cokernel of the map $i^*i_*: CH^{j-1}(\pone \times {\mathbb P}^3) \to CH^j(\pone \times {\mathbb P}^3)$, which comes from the intersection with a divisor formula \cite[Proposition 2.6c]{Fulton}.

By definition $[Z] \in CH^1({\pone} \times {\mathbb P}^3)$ is precisely $d_1 \xi + d_2 \mu$.  The first isomorphism of the statement then follows immediately from this: $CH^1(X) \cong \Z \xi \oplus \Z \mu / \langle d_1 \xi + d_2 \mu \rangle$.  For the second isomorphism we proceed as follows. Consider the matrix representing $i^*i_*$ with respect to the bases $\xi,\mu$ of $CH^1(\pone \times {\mathbb P}^3)$ and $\xi\mu, \mu^2$ of $CH^2(\pone \times {\mathbb P}^3)$.  This matrix is not diagonal, but can be put in Smith normal form via the identity:
\[
\begin{pmatrix}
1 & 0 \\
-\frac{md_2}{g} & 1
\end{pmatrix}
\begin{pmatrix}
d_2 & d_1 \\
0 & d_2
\end{pmatrix}
\begin{pmatrix}
n & -\frac{d_1}{g} \\
m & \frac{d_2}{g}
\end{pmatrix}
= \begin{pmatrix}
g & 0 \\
0 & \frac{d_2^2}{g}
\end{pmatrix}
\]
In particular, the cokernel of the map $i^*i_*$ can be computed from that of $diag(g,\frac{d_2^2}{g})$, which is what was asserted in the statement.  The statements regarding the images of $\xi \mu$ and $\mu^2$ are immediate, and the final statement follows from the choice of B\'ezout identity presenting $gcd(d_1,d_2)$ that arises from the Euclidean algorithm.
\end{proof}

\begin{thm}
\label{thm:mainexample}
Suppose $Z \subset \pone \times {\mathbb P}^3$ is a smooth complex hypersurface of bidegree $(3,4)$ defined over $\overline{\Q}$ that specializes modulo some prime $p$ to the singular hypersurface $y_0^3 x_0^4 + y_0^2y_1 x_1^4 + y_0 y_1^2 x_2^4 + y_1^3 x_3^4$ over $\overline{\mathbb{F}}_p$.  The classes $\xi \mu$ and $\xi \mu^2$ both restrict non-trivially from $\pone \times {\mathbb P}^3$ to $X = (\pone \times {\mathbb P}^3) \setminus Z$ and if $\psi$ is the image of $\xi \mu \in CH^2(X)$, then $Sq^2 \psi \neq 0 \in CH^3(X)/2$.
\end{thm}

\begin{proof}
Observe that $Sq^2(\xi \mu) = \xi \mu^2 + \xi^2 \mu = \xi \mu^2$ in $CH^*(\pone \times {\mathbb P}^3)$ by the Cartan formula \cite[Proposition 9.7]{VRed} (note that $Sq^1(\xi) = Sq^1(\mu) = 0$ since $H^{3,1}(W,\Z)$ vanishes for any smooth scheme $W$).  Since motivic Steenrod operations are compatible with pullbacks along morphisms of smooth schemes by construction, the statement that $Sq^2(\psi)$ is non-trivial will follow from the assertion that $\xi \mu$ and $\xi \mu^2$ restrict non-trivially to $X$.  Since $gcd(3,4) = 1$, the fact that $\xi \mu$ restricts non-trivially to $X$ follows immediately from Proposition \ref{prop:restrictiondegree2nontrivial}.  Thus, it remains to show that $\xi \mu^2$ restricts non-trivially to $X$.

To this end, we use an idea of Totaro who showed that every curve in $Z$ has even degree over $\pone$.  Indeed, see \cite[Proof of Theorem 3.1]{Totaro} for this precise statement.  Granted this statement, one deduces immediately that $i_*[C]$ for a $1$-cycle $C$ on $Z$ is of the form $2 r \xi \mu^2 + s \mu^3$.  In particular, $\xi \mu^2$ does not lie in the image of $i_*$ or the corresponding map for Chow groups modulo $2$ and therefore restricts non-trivially to $CH^3(X)/2$. (Note: this last observation reproves the fact that $\xi\mu$ restricts non-trivially to $X$ since it follows immediately from the previous statement that $Sq^2(\psi)$, which is the restriction of $Sq^2(\xi \mu) = \xi \mu^2$ to $X$, is non-trivial.)


\end{proof}

\begin{cor}
\label{cor:mainexample}
Suppose $X$ is a variety as described in the statement of \textup{Theorem \ref{thm:mainexample}}.  For a given $c_2^{top} \in H^4(X^{an},\Z)$, there is a unique topological vector bundle $\mathcal{E}^{an}$ over $X^{an}$ with $c_1^{top}(\mathcal{E}^{an}) = 0$ and $c_2^{top}(\mathcal{E}^{an}) = c_2^{top}$.  If $c_2^{top}$ is the image of $\xi \mu$ under the cycle class map, then the topological vector bundle with Chern classes $(0,c_2^{top})$ has algebraic Chern classes yet fails to be algebraizable.
\end{cor}

\begin{proof}
As discussed in the proof of Theorem \ref{thmintro:maintheorem} at the end of Subsection \ref{ss:identifyingobstructions} the map $(c_1^{top},c_2^{top}): \mathscr{V}_2^{top}(X) \to H^2(X^{an},\Z) \times H^4(X^{an},\Z)$ is a pointed bijection.  Thus, the pair $(0,c_2^{top})$ determines a unique rank $2$ topological vector bundle on $X^{an}$.  In the case where $c_2^{top}$ is the image of $\xi \mu$ under restriction, it follows immediately from Theorem \ref{thm:mainexample} that $Sq^2(\xi\mu) \neq 0 \in CH^3(X)/2$.  Therefore, the claim follows from Theorem \ref{thm:main}.
\end{proof}

\subsection{Conjectures on cycles and genericity of the examples}
\label{ss:noritotaro}
The example of the previous section seems intimately related to the failure of the integral Hodge conjecture.  We explore this connection in greater detail now.

\begin{ex}
\label{ex:trento}
There are examples due to Koll{\'a}r--van Geemen that show the integral Hodge conjecture can fail for hypersurfaces in projective space \cite{Trento} or \cite[Theorem 2]{SouleVoisin}.  Indeed, suppose $Z \subset {\mathbb P}^4$ is a hypersurface of degree $d$.  Assume that for some integer $p$ coprime to $6$, $p^3$ divides $d$.  Then, for a general $Z$, any curve $C \subset Z$ has degree divisible by $p$.  Observe that, in this case, $Pic(Z) \cong \Z$ by the Grothendieck--Lefschetz theorem.  Setting $X := {\mathbb P}^4 \setminus Z$, we can compute $CH^i(X)$ in low degrees.  Identify $CH^*({\mathbb P}^4) = \Z[\xi]/\xi^5$.  Then, $CH^1(X) = \Z/d\Z$ generated by the image of $\xi$ and $CH^2(X) = \Z/d\Z$ generated by the class of $\xi^2$.  Note that $Sq^2(\xi^2) = 2 \xi Sq^2 \xi = 0 \in CH^3(X)/2$.  Take a topological complex vector bundle on $X$ of rank $2$ with Chern classes of the form $(m\xi,a\xi^2)$.  If $m$ is even, all such bundles are necessarily algebraizable.  If $m$ is odd, then $Sq^2(a \xi^2) + m \xi \cup a \xi^2 = am \xi^3$.  If $a$ is also odd, then $am$ is odd and so $am\xi^3$ {\em could} restrict non-trivially to $CH^3(X)/2$.  However, the construction above only shows that $CH^3(X)$ is a quotient of $\Z/n\Z$ where $p|n$.  However, since $p$ is odd, we do not know whether $\xi^3$ restricts non-trivially to $CH^3(X)/2$.
\end{ex}

\begin{ex}
By \cite[Lemma 5.1]{Totaro}, there is a smooth hypersurface $Z \subset {\mathbb P}^4$ of degree $48$ over $\bar{\Q}$ for which the integral Hodge conjecture fails.  The construction of this example is, however, rather involved.  Set $X = {\mathbb P}^4 \setminus Z$.  In this case, $Pic(X) \cong \Z/48\Z$ and $CH^2(X) \cong \Z/48\Z$, generated by $\xi$ and $\xi^2$ in the notation of Example \ref{ex:trento}.  Totaro shows that every curve $C \subset X$ has even degree over $\bar{\Q}$, i.e., the pushforward map $CH^2(Z) \to CH^3(X)$ has image contained in $2\xi^3$.  In particular, in this example $CH^3(X)/2$ is necessarily non-trivial, and $\xi^2$ restricts non-trivially.  As before, if we fix Chern classes $(m \xi,a\xi^2)$ with both $a$ and $m$ odd, then $ma \xi^3$ restricts non-trivially to $CH^3(X)/2$ and we thus obtain more non-algebraizable bundles with algebraic Chern classes.
\end{ex}

Let us recall the following conjecture of Nori, as modified by Totaro.

\begin{conj}[Nori,Totaro]
\label{conj:noritotaro}
If $Y$ is a smooth projective variety, and $Z \subset Y$ is a very general, sufficiently ample hypersurface, then the restriction map $CH^i(Y) \to CH^i(Z)$ is an isomorphism for $i < \dim Z$.
\end{conj}

Combined with Theorem \ref{thmintro:maintheorem}, Conjecture \ref{conj:noritotaro} suggests that non-algebraizable topological complex vector bundles of rank $2$ should be rather common.

\begin{ex}
Under the hypotheses of Conjecture \ref{conj:noritotaro}, set $\xi = [Z] \in Pic(Y)$ and $X := Y \setminus Z$, which is necessarily affine.  Observe then that if $Z \subset Y$ has dimension $3$, then $Pic(X) = Pic(Y)/\langle \xi \rangle$, $CH^2(X) = CH^2(Y)/ \langle \xi \cup Pic(Y) \rangle$, while $CH^3(X) = CH^3(Y)/\langle \xi \cup CH^2(Y) \rangle$.  Thus, assuming Nori's conjecture, if $Y = {\mathbb P}^4$ then for $d$ sufficiently large, and any sufficiently general hypersurface of degree $d$, then $Pic(X) \cong \Z/d\Z$, $CH^2(X) \cong \Z/d\Z$ and $CH^3(X) \cong \Z/d\Z$ generated by the image of $\xi,\xi^2,\xi^3$.  If $d$ is even, then we expect only those topological complex vector bundles of rank $2$ whose Chern classes are of the form $(m\xi,a\xi^2)$ with $a$ and $m$ both odd to be non-algebraizable.
\end{ex}

\bibliographystyle{alpha}
\bibliography{algebraicity}

\end{document}